\documentclass[draft]{birkjour}
\usepackage[noadjust]{cite}
\usepackage{xcolor}
\RequirePackage[all]{xy}

\usepackage[utf8]{inputenc}
\usepackage{amsmath, amssymb, amsfonts}
\usepackage{hyperref}
\usepackage{geometry}
\usepackage{tikz-cd}
\newtheorem{Theorem}{Theorem}[section]
\newtheorem{Cor}[Theorem]{Corollary}
\newtheorem{Lemma}[Theorem]{Lemma}
\newtheorem{Proposition}[Theorem]{Proposition}
\theoremstyle{definition}
\newtheorem{Definition}[Theorem]{Definition}
\theoremstyle{remark}
\newtheorem{rem}[Theorem]{Remark}

\numberwithin{equation}{section}

\newcommand{\R}{\mathbb R}
\newcommand{\N}{\mathbb N}

\newcommand{\K}{\mathbb{K}}

\newcommand{\p}{\mathcal{P}}
\newcommand{\A}{\mathcal{A}}

\newcommand{\rel}{\mathcal{R}}

\newcommand{\mbc}{{\mathbb C}}

\newcommand{\BibTeX}{B\kern-0.1emi\kern-0.017emb\kern-0.15em\TeX}
\newcommand{\XYpic}{$\mathrm{X\kern-0.3em\raisebox{-0.18em}{Y}}$-$\mathrm{pic}\,$}

\newcommand{\cl}{C \kern -0.1em \ell}  

\newcommand{\ed}{\end{document}}

\title{Diffeological Generalized Formal Series: An Overview}
\author{Jean-Pierre Magnot}
\address{Univ. Angers, CNRS, LAREMA, SFR MATHSTIC, F-49000 Angers, France ;\\ 
Lepage Research Institute, 17 novembra 1, 081 16 Presov, Slovakia;
	\\ and \\  Lyc\'ee Jeanne d'Arc,  Avenue de Grande Bretagne, 63000 Clermont-Ferrand, France}
\email{magnot@math.cnrs.fr}
\thanks{The author thanks the organizers of the XLII Workshop on Geometric Methods in Physics for their kind invitation and support. He also thanks Anahita Eslami-Rad and Enrique G. Reyes for deep exchanges on generalizations of series during his collaborations on integrable systems, Ambar Sengupta and Sergio Albeverio for stimulating questions and exchanges on stochastic cosurfaces, Hong Van Le for her long and helpful informal exchanges on tensor series and Amandine V\'eber for short but enlightening exchanges on branching processes in morphogenesis.
Finally, the author acknowledges the France 2030 framework programme Centre Henri Lebesgue ANR-11-LABX-0020-01 for 
creating an attractive mathematical environment. }
\begin{document}

\begin{abstract}
We present an overview of some recent developments in the theory of generalized formal series, grounded in diffeological geometric framework. These constructions aim to offer new tools for understanding infinite-dimensional phenomena in mathematical physics, particularly in contexts where standard atlases and charts are insufficient. Applications to tensor structures and categorical gradings are discussed, and emerging applications are mentioned.
\end{abstract}

\subjclass{18F40, 13F05, 58A05}

\keywords{ generalized formal series, diffeologies, infinite dimensional geometry, non-manifolds differential structures.} 

\maketitle
\section{Introduction}

Formal series are ubiquitous in mathematics and physics, notably in the study of asymptotic expansions, deformation quantization, and integrable systems. While classical formal series are defined over standard algebraic structures, many modern problems—especially those involving infinite-dimensional objects—require more flexible geometric frameworks. 

This article explores generalized formal series built upon diffeological and Frölicher spaces, allowing for novel algebraic and geometric constructions beyond the classical setting.

\section{Diffeological preliminaries} \label{s:diffeology}
	This section provides the necessary background on diffeologymostly following 
	\cite{GMW2024}, but keeping \cite{Igdiff} as our main reference for diffeologies remains. A 
	complementary non-exhaustive bibliography on these subjects is 
	\cite{BKW2025,Wa}. {{} The main idea of diffeologies is to choose as a basement of definition of smooth intrinsic  objects a set of smooth mappings (with well-chosen compatibility conditions) instead of an atlas, considering manifolds as a restricted class of examples. We choose diffeologies because they possess 
	nice properties such as cartesian closedness, thereby carrying the necessary fundamental properties for e.g. 
	calculus of variations, and also because they are very easy to use in a differential geometric way of 
	thinking. 
	
	\begin{Definition}[Diffeology] \label{d:diffeology}
		Let $X$ be a set.  A \textbf{parametrisation} of $X$ is a
		map 
		$p \colon U \to X$, in which $U$ is an open subset of Euclidean space (no fixed dimension).  A 
		\textbf{diffeology} $\p$ on $X$ is a set of parametrisations satisfying the following three conditions:
		\begin{enumerate}
			\item (Covering) For every $x\in X$ and every non-negative integer
			$n$, the constant function $p\colon \R^n\to\{x\}\subseteq X$ is in
			$\p$.
			\item (Locality) Let $p\colon U\to X$ be a parametrisation such that for
			every $u\in U$ there exists an open neighbourhood $V\subseteq U$ of $u$
			satisfying $p|_V\in\p$. Then $p\in\p$.
			\item (Smooth Compatibility) Let $p\colon U\to X$ be a parametrisation in $\p$.
		Then for every $n >0$, every open subset $V\subseteq\R^n$, and every
			smooth map $F\colon V\to U$, we have $p\circ F\in\p$.
		\end{enumerate}
		
		A \textbf{diffeological space} $(X,\p)$ is a set $X$ equipped with a diffeology $\p$.
		When the diffeology is understood, we drop the symbol $\p$.
		The parametrisations $p\in\p$ are called \textbf{plots}.
	\end{Definition}
	
	{{}
		\noindent\textbf{Notation.} We recall that $\N^* = \{n \in \N \, | \, n \neq 0\}$ and that 
		$\forall m \in \N^*, \N_m = \{1,...,m\} \subset \N.$}
	
	\begin{Definition}[Diffeologically Smooth Map]\label{d:diffeolmap}
		Let $(X,\p_X)$ and $(Y,\p_Y)$ be two diffeological
		spaces, and let $F \colon X \to Y$ be a map.  We say that $F$ is
		\textbf{diffeologically smooth} if for any plot $p \in \p_X$,
		$F \circ p \in \p_Y.$
	\end{Definition}
	
	\noindent
	Diffeological spaces with diffeologically smooth maps form a category which is complete, co-complete, 
	and in fact, a quasi-topos (see \cite{BH}).  In particular, in this category pull-backs, push-forwards, 
	(infinite) products, and (infinite) coproducts exist. Diffeological spaces also carry a natural topology called $D-$topology, which is obtained by push-forward topology along the plots.
	The recent review \cite{GMW2024} fully describes these constructions.

	\begin{Definition}
		Let $(X,\p)$ and $(X',\p')$
		be two diffeological spaces.  A map $f : X \rightarrow X'$ is  called  a subduction
		if  $\p' = f_*(\p).$ 
	\end{Definition}

	\begin{Definition}[Subset Diffeology]\label{d:diffeol subset}
		Let $(X,\p)$ be a diffeological space, and let $Y\subseteq X$.  Then $Y$ comes equipped with the 
\textbf{subset diffeology}, which is the set of all plots in $\p$ with image in $Y$.
	\end{Definition}

\noindent
	If $X$ is a smooth manifold modelled on a complete locally convex topological 
	vector space, we define its \textbf{nebulae diffeology} as
\begin{equation} \label{aste}
	\p_\infty(X) = \left\{ p \in C^\infty(O,X) \hbox{ (in the usual sense) }| O \hbox{ is open in } \R^d, d \in \N^* 
	\right\}.
\end{equation}

	\begin{rem} \label{comp}
		Diffeological, $c^\infty$ and G\^ateaux smoothness are the same notion
		if we restrict to a Fr\'echet context, see e.g. \cite{GMW2024,BKW2025}.In particular, for manifolds modeled on Banach, Fr\'echet, or even locally convex vector spaces, classical smoothness is equivalent to diffeological smoothness for the underlying nebulae diffeologies. 
	\end{rem}
	
	\begin{Proposition} \label{quotient} 
	Let $(X,\p)$ b a diffeological	space and $\rel$ an equivalence relation on $X$. Then, there is
		a natural diffeology on $X/\rel$ defined as the push-forward diffeology by the quotient projection
		$X\rightarrow X/\rel,$ denoted by $\p/\rel$.
	\end{Proposition}
	
	Within these constructions, algebraic structures can be \emph{diffeological} if their algebraic laws are smooth within the underlying diffeologies. For example, 
	 $V$ is a diffeological vector space if it is an algebraic vector space $(V,+,.)$ over a scalar field $(\mathbb{K},+,.)$ such that \cite[Chapter 3]{Igdiff}: 
	\begin{itemize}
		\item $V$ and $\mathbb{K}$ are diffeological spaces; 
		\item the addition in $V$ is a smooth map $V\times V \rightarrow V$;
		\item the addition and the multiplication in $\mathbb{K}$ are smooth maps $\mathbb{K} \times \mathbb{K} \rightarrow \mathbb{K}$;
		\item $\mathbb{K}^*$ is open in $\mathbb{K}$ for the D-topology, and the inversion map $\mathbb{K}^* \rightarrow \mathbb{K}^*$ is smooth;  
  for $\mathbb{K}=\R$ or $\mathbb{C},$ the D-topology is the standard topology;
		\item scalar multiplication $\mathbb{K} \times V \rightarrow V$ is smooth;
  \item  $\R  \subset \mathbb{K}$ is a diffeological subfield.
	\end{itemize}

%

\section{Generalized formal series}

\subsection{Diffeological formal series over a diffeological field}

Let $k$ be a diffeological field in the sense of \cite{ERMR}, i.e., a field equipped with a diffeology compatible with the field operations. 
We define the set of formal power series $k[[x]]$ as the set of sequences $(a_n)_{n\in\mathbb{N}}$ with $a_n\in k$. This space inherits a natural diffeology via the identification:
\(
k[[X]] \simeq k^{\mathbb{N}},
\)
where $k^{\mathbb{N}}$ is endowed with the product diffeology. This choice ensures smooth evaluation of truncated series and continuity of algebraic operations. A typical example is $\mathbb{R}$ or $\mathbb{C}$ endowed with their standard diffeology, or more exotic constructions like the field of $\mbc-$Laurent series with the infinite product diffeology.

\begin{rem} In contrast to the classical setting, where only the algebraic structure is emphasized, the diffeological approach ensures compatibility with smooth structures. \end{rem}
\subsection{Completion of a diffeological ultrametric space}

Let $(V, d)$ be a vector space over a valued field $(k, |\cdot|)$, where $d(x,y) = |x - y|$ defines an ultrametric distance. Suppose $V$ is equipped with a diffeology compatible with the ultrametric topology.
In \cite{ERMR}, the authors define a diffeological completion $\widehat{V}$ of $V,$ without additional assumptions. This yields a complete diffeological vector space, with natural inclusion $V \hookrightarrow \widehat{V}$ smooth and dense. Moreover, the construction is functorial: any smooth linear map between such spaces extends uniquely to their completions.It has been useful in \cite{ERMR} in the study of the r-matrix decomposition of the Kadomtsev-Petviashvili hierarchy in a non-ultrametric, generalized setting. In this example, the ultrametric distance generalize the valuations of the $(\partial,t_1,t_2,\cdots)-$series.

\subsection{Tensor-valued series}

	\subsubsection{Diffeologies on  power series of algebraic tensor product algebras}

 In what follows, we assume that the diffeological dual $V'$ of $V$ separates $V,$ that is, \(\forall (v,w)\in V^2, (v \neq w) \Leftrightarrow (\exists \alpha \in V', \alpha(v) \neq \alpha(w)).\)   Let us denote by $B(V,W)$ the  \color{blue} \color{black} space of smooth  bilinear forms on $V$ and $W$,  which   is endowed  with the standard  functional  diffeology,  see \cite[\S 1.57, p. 34]{Igdiff}. When $V$ and $W$ are two diffeological spaces, we define, generalizing \cite{BS12}, the tensor product $V\otimes W$ as the (algebraic) linear span of the image of $V \times W$ in the diffeological dual $(B(V,W))'$ of $B(V,W)$ by the evaluation mapping 
$(v,w) \in V \times W \mapsto (b \in B(V,W) \mapsto b(v,w)).$ Within this construction, for $(\lambda,\mu,v_1,v_2,w_1,w_2)\in \mathbb{K}^2\times V^2 \times W^2$, $\forall b \in B(V,W),$ 
$b(\lambda v_1 + \mu v_2 , w_1)-\lambda b(v_1,w_1)-\mu b(v_2,w_1)=0$ { and }  $b(v_1,\lambda w_1 + \mu w_2)-\lambda b(v_1,w_1)-\mu b(v_1,w_2)=0.$ 

If $V$ and $W$ are diffeological vector spaces, we shall equip $V \otimes W$  with the diffeology compatible  with  vector space  with  the push-forward diffeology     there  is no  push forward map $V \times W \to   V \otimes W$ with respect to the bilinear map $\otimes \, : \,  (v,w)\in V\times W \mapsto V \otimes W.$ 
\begin{Definition} Within this definition, a plot $p$ on $V \otimes W$ is a map from an open subset $D(p)$ of an Euclidean space to $V \otimes W$ such that, $\forall x \in D(p), \exists U$ { open neighborhood of } $x$ { in } $O, \exists ((v_i,w_i))_{i \in \N_n} \in (C^\infty(U,V)\times C^\infty(U,W))^n,$  $p|_U = \sum_{k=1}^n v_k \otimes w_k
.$
\end{Definition} \label{remark1.1}

\begin{Lemma} \label{lemma1.2}
    The addition $V \otimes W \times V \otimes W \rightarrow V \otimes W$ is smooth.
\end{Lemma}
  \begin{proof}
      Any plot in $V \otimes W \times V \otimes W$ reads locally as $p\times q,$ where $p$ and $q$ are plots in $V \otimes W.$ {More precisely, $p \times q$ is a (local) plot on $V \otimes W \times V \otimes W$ such that the plots $p$ and $q$ have a common domain $U$ in an Euclidean vector space. Therefore, $p$ is}
       such that $\exists ((v_i,w_i))_{i \in \N_n} \in (C^\infty(U,V)\times C^\infty(U,W))^n,$ $p = \sum_{k=1}^n v_k \otimes w_k.$ The same way, let $q$ be such that $\exists ((v'_i,w'_i))_{i \in \N_n} \in (C^\infty(U,V)\times C^\infty(U,W))^m,$ $q = \sum_{l=1}^m v'_l \otimes w'_l .$
      Therefore, the image of { $p \times q$ }reads as $$\left(\sum_{k=1}^n v_k \otimes w_k\right) + \left(\sum_{l=1}^m v'_l \otimes w'_l\right), $$ which is sufficient to show that it defines a smooth plot on $V \otimes W$.
  \end{proof}  

\begin{Lemma} \label{Lemma1.3}
The scalar multiplication  $\mathbb{K} \times  V \otimes W \rightarrow V \otimes W$ is smooth.
\end{Lemma}
\begin{proof}
{A plot on $\mathbb{K} \times V \otimes W$ reads locally as $\lambda \times p,$ where $\lambda$ is a plot on $\mathbb{K},$ $p$ is a plot on $V \otimes W,$ and $\lambda$ and $p$ have the same domain $U$ in and Euclidean space. The plot $p$ can also be assumed to read as }
      $p = \sum_{k=1}^n v_k \otimes w_k,$ where $((v_i,w_i))_{i \in \N_n} \in (C^\infty(U,V)\times C^\infty(U,W))^n.$
    Then \(
        \lambda \otimes p  =  \lambda \otimes \left(\sum_{k=1}^n v_k \otimes w_k\right)  =  \sum_{k=1}^n (\lambda v_k) \otimes w_k.
    \)
     Therefore, {$\lambda \otimes p \in C^\infty(U, V\otimes W).$} This is sufficient to prove that scalar multiplication is smooth. 
\end{proof}
\noindent
As a consequence of Lemma \ref{lemma1.2} and of Lemma \ref{Lemma1.3}, we have: 
\begin{Proposition} \label{proposition1.4}
    Let $V$ and $W$ be two diffeological vector spaces over a diffeological field $\mathbb{K}.$ Then $V \otimes W$ is a diffeological vector space over $\mathbb{K}.$
\end{Proposition}

\begin{Lemma} \label{lemma1.5}
    The canonical Cartesian product map
$ V^{\times n} \times V^{\times m} \rightarrow V^{\times(n+m)}$ reduces  to a  smooth map $\otimes:$ 
$ V^{\otimes n} \times V^{\otimes m} \rightarrow V^{\otimes(n+m)}$
\end{Lemma} 

\begin{proof}
     {A plot on $ V^{\times n} \times V^{\times m}$ locally reads as $p^{(n)} \times p^{(m)}$, where }$p^{(n)}$ is a plot on $V^{\otimes n}$ and  $p^{(m)}$ is a plot on $V^{\otimes m},$ both defined on the same domain $U$ of an Euclidean space.   Following Remark \ref{remark1.1},  $$\exists ((v_j^i)_{(i,j) \in \N_n \times \N_a},(w_j^i)_{(i,j) \in \N_m \times \N_b} \in (C^\infty(U,V))^{na}\times C^\infty(U,M))^{mb},  $$ 
    $$p^{(n)} = \sum_{k=1}^a \otimes_{p=1}^n v^p_k
    \hbox{ and } p^{(m)} = \sum_{l=1}^b \otimes_{q=1}^n w^q_l.$$
    Then $p^{(n)} \otimes p^{(m)}= \sum_{k=1}^a \sum_{l=1}^b (\otimes_{p=1}^n v^p_k)\otimes (\otimes_{q=1}^n w^q_l)$ is a smooth plot on $V^{n+m}$ and this check is sufficient to prove that $\otimes$ is smooth.
\end{proof}
\noindent
For $n \in \N^*,$ the n-th tensor product of $V$ is defined algebraically by induction as the space $V^{\otimes n+1} = V \otimes V^{\otimes n},$
and the $n$-tensor product maps $(v_i)_{i \in \N_n} \in V^n \mapsto \bigotimes_{i \in \N_n} v_i \in V^{\otimes n}$ are $n$-multilinear.  We equip the $n$-th tensor product with the push-forward diffeology too. 

\begin{Lemma} \label{lemma1.5}
    Let $V$ be a diffeological vector spaces over a diffeological field $\mathbb{K}.$ Then $\forall n \in \N^*,$ $V^{\otimes n}$ is a diffeological vector space over $\mathbb{K}$.
\end{Lemma}

\begin{proof}
    By induction, by applying Proposition \ref{proposition1.4} to $V^{\otimes n+1}= V^{\otimes n} \otimes V.$
\end{proof}

 Lemma \ref{Lemma1.3} implies

\begin{Theorem} \label{tensor diffeological group}
	The power series  tensor algebra $(T((V)),+,\otimes)$ defined by 
	$$T^\otimes((V)) = \mathbb{K} \oplus \prod_{k \in \N^*} V^{\otimes k}$$
	is a diffeological $\mathbb{K}$-algebra for the infinite product diffeology, whose  group of the units $T^\otimes((V))^{\times}$  is   $T_0((V)) - \{0\}$,  which  is a diffeological group for the subset diffeology.  \end{Theorem} 

\begin{rem}
    In general, when $A$ is a diffeological algebra, the natural diffeology of $A^*$ can be different from the subset diffeology of $A,$ see e.g. \cite{GMW2024} for discussions on the diffeological algebra $L(V)$ of smooth linear endomorphisms of a diffeological vector space $V.$ Indeed, in general, we are not sure that the inversion map $a \mapsto a^{-1}$  is smooth for the dubset diffeology. 
\end{rem}

The space $T^\otimes((V))$ is naturally graded by the order $n$ of the tensor product $V^{\otimes n}$.  We note by $[a]_n$ the $V^{\otimes n}$-component of $ a \in T^\otimes((V)).$ Therefore, we can define a degree
and a valuation 
.

\begin{proof}[Proof of Theorem \ref{tensor diffeological group}]
   As a direct consequence of Lemma \ref{lemma1.2}, Lemma \ref{Lemma1.3} and of Lemma \ref{lemma1.5}, we have that $T^\otimes((V))$ is a diffeological algebra, that is, an algebra and a diffeological space for which addition, multiplication and scalar multiplication are smooth. 
In order to prove that inversion is smooth, we check the classical algebraic formulas for inversion. Let $T^{\otimes,>0}((V))$ be the two-sided ideal
   $T^{\otimes,>0}((V))= \prod_{k \in \N^*} V^{\otimes k} = \left\{ a \in T^\otimes((V)) \, | \, val^{\otimes}(a)>0\right\}.$
   
   Let $\lambda + t \in \mathbb{K} \oplus T^{\otimes,>0}((V)).$ Then $\lambda + t$ is invertible if and only if $\lambda \neq 0$ and 
   $ (\lambda + t)^{-1} = \lambda^{-1} \left(1 + \sum_{k=1}^{\infty} (-1)^k \left(\lambda t\right)^k \right).$
   All these operations are smooth therefore inversion is smooth. 
\end{proof}

\subsubsection{Exponential, logarithm and formal series}
We now settle $\K = \R$ or $\mathbb{C}.$

\begin{Definition} Let $\lambda + u \in \K + T^{\otimes,>0}((V)) = T^\otimes((V))$ and we define
$ \exp\left(\lambda + u\right) = \exp (\lambda) \left( \sum_{k=0}^{+\infty} \frac{1}{k!}u^k\right)$
and, for \( \lambda \in \R_+^*, \,  \log\left(\lambda + u\right) = \log (\lambda) + \left( \sum_{k=1}^{+\infty} \frac{(-1)^{k+1}}{k}\left(\frac{u}{\lambda}\right)^k\right).\)
\end{Definition}
\noindent
In the sequel we assume that $\mathbb{K}=\R .$
\begin{Theorem}
    The group $ T^\otimes((V))^\times$  has two connected components:
    $$T^{\otimes,\pm}((V)) = \left\{ \lambda + u \in \R + T^{\otimes,>0}((V))  \, | \, \lambda \in \R_\pm^* \right\}.$$
    Moreover, $\exp:  T_0((V)) \rightarrow T_0^+((V)) $ is an isomorphism of diffeological group with smooth inverse $\log$.
\end{Theorem}
\begin{proof}
    The algebraic part is elementary and well-known.
    Let us consider the projection $[.]_0 , $ which is smooth. We have that  $T_0^+((V))=[.]^{-1}_0(\R_+^*)$ and $T_0^-((V))=[.]^{-1}_0(\R_-^*)$ are both open and closed in the $D-$topology. They are also convex cones, thus they are connected.
The maps  $\exp$ and $\log$ are both defined by smooth operations,hence they are smooth.
\end{proof}
\noindent
We can now define the Baker-Campbell-Hausdorff series.
Next theorem is straightforward:
\begin{Theorem}
     $BCH: (u,v) \mapsto \log(\exp(u)\exp(v)).$ is smooth.
\end{Theorem}

\subsubsection{Symmetrized tensors}
 Let us first fix $n \in \N^*.$ 
 { Let us note by $S^n(V)$ the space of symmetric diffeological n-linear forms on $V.$ When $V$ and $W$ are two diffeological spaces, we define, generalizing \cite{BS12}, the symmetrized tensor product $V^{\odot n}$ as the image of $V^{\times n}$ in the diffeological dual $(S^n(V))'$ of $S^n(V)$ by the evaluation mapping
$(v_1,...v_n) \in V^{\times n} \mapsto (b \in S^n(V) \mapsto b(v_1,...,v_n)).$
Let us fix $n \in \N^*.$ We notice that $S^n(V)$ is a (diffeological) linear subspace of the space of diffeological $n-$linear maps $B^n(V).$ The symmetrization map 
\begin{equation} \label{sym1} \hbox{sym}_n:b \in B^n(V) \mapsto \left( (v_1,...,v_n) \mapsto \frac{1}{n!} \sum_{\sigma \in \mathfrak{S}_n} b(v_{\sigma(1)},...,v_{\sigma(n)} )\right)\end{equation} is a smooth projection  $\hbox{sym}_n: B^n(V) \rightarrow S^n(V) $
while the canonical inclusion $S^n(V) \hookrightarrow B^n(V)$ the smooth left inverse of $\hbox{sym}_n.$ The same way, we define:
\begin{Definition}
The symmetrization mapping in $V^{\otimes n}$ is the the restriction to $V^{\otimes n}$ of dual map $B^n(V)' \rightarrow S^n(V)'$ of the canonical inclusion $S^n(V) \hookrightarrow B^n(V).$ We also denote this mapping by $\hbox{sym}_n$ since it carries no ambiguity.
\end{Definition}
In order to get a more precise idea of $\hbox{sym}_n$ acting on $V^{\otimes n},$ we apply in a standard way equation \ref{sym1} to get \begin{equation} \label{sym2} \hbox{sym}_n(v_1 \otimes... \otimes v_n) = \frac{1}{n!} \sum_{\sigma \in \mathfrak{S}_n} v_{\sigma(1)} \otimes... \otimes v_{\sigma(n)}:= v_1 \odot ... \odot v_n.  \end{equation}
We define a smooth symmetrized tensor product
$ \odot : V^{\odot n} \times V^{\odot m} \rightarrow V^{\odot (n+m)}$
by checking that the following commutative diagram is made of smooth maps:

\vskip 12pt
\begin{center}
\begin{tikzcd}
V^{\times n} \times V^{\times m}\arrow[r, "\otimes_n \times \otimes_m"] \arrow[d]
& V^{\otimes n} \times V^{\otimes m} \arrow[r, "\hbox{sym}_n \times \hbox{sym}_m"] \arrow[d, "\otimes" ] & V^{\odot n} \times V^{\odot m} \arrow[d, "\odot"]\\
V^{\times (n + m)} \arrow[r,  "\otimes_{n+m}" ]
& V^{\otimes (n + m)} \arrow[r,  "\hbox{sym}_n" ] & V^{\odot (n + m)}
\end{tikzcd}
\end{center}
\vskip 12pt

\begin{Theorem} \label{sym-tensor diffeological group}
	The symmetric power series  tensor algebra $(T^\odot((V)),+,\odot)$ defined by 
	$T^\odot((V)) = \mathbb{K} \oplus \prod_{k \in \N^*} V^{\odot k}$
	is an abelian diffeological $\mathbb{K}$-algebra,    whose  group of the units $T^\odot((V))^{\times} = T^\odot((V) \cap T^\otimes((V))^\times$  is  $T^\odot((V)) - \{0\}$,   which  is an abelian diffeological group.  \end{Theorem} 

\begin{rem} \label{remark-degree-sym}
The space $T_0^\odot((V))$ is graded by the order $n$ of the tensor product $V^{\odot n}$ as a diffeological vector supbspace of $V^{\otimes n}$.  We therefore restrict the notations of the degree and of the valuation to $T^\odot((V)). $
\end{rem}

\noindent
We now assume that $\K = \R$ or $\mathbb{C}.$ As for the (non symmetrized) tensor product $\otimes.$ By the same constructions:
\( \exp^\odot\left(\lambda + u\right) = \exp (\lambda) \left( \sum_{k=0}^{+\infty} \frac{1}{k!}u^{\odot k}\right),\)
and if $\K=\R,$
 $\lambda \in \R_+^*,$
\( \log^\odot\left(\lambda + u\right) = \log (\lambda) + \left( \sum_{k=1}^{+\infty} \frac{(-1)^{k+1}}{k}\left(\frac{u}{\lambda}\right)^{\odot k}\right).\)
we get, with the same arguments as for $\otimes,$
\begin{Theorem}
     $ T^\odot((V))^\times$  has two connected components:
    $T^{\odot,\pm}((V)).$
    Moreover, $\exp^\odot:  (T^\odot((V)),+) \rightarrow (T^{\odot,+}((V)),\odot) $ is an isomorphism of abelian diffeological groups with smooth inverse $\log^\odot$.
\end{Theorem}

\subsubsection{The symmetrization bundle}

Let us concentrate on $sym: T^\otimes((V))^\times \rightarrow T^\odot((V))^\times$
{ and on its analog } $sym: T^\otimes((V)) \rightarrow T^\odot((V)).$

\begin{Proposition}
   $sym: T_0^\otimes((V))^\times \rightarrow T_0^\odot((V))^\times$
and $sym: T_0^\otimes((V)) \rightarrow T_0^\odot((V))$ are smooth morphisms of groups and of algebras, respectively. 
\end{Proposition}
\noindent
Let $K = Ker (sym) \subset T_0^\otimes((V))^\times$ and let $\mathfrak{K} = Ker (sym) \subset T_0^\otimes((V)).$
\begin{Proposition}
    $T_0^\odot((V))^\times = T_0^\otimes((V))^\times / K$ and $T_0^\odot((V)) = T_0^\otimes((V)) / \mathfrak{K}.$
\end{Proposition}
Therefore we get a short exact sequence of diffeological Lie groups
$$ 1 \rightarrow K \rightarrow T_0^\otimes((V))^\times \rightarrow T_0^\odot((V))^\times \rightarrow 1$$
with inclusion map which is a global section $T_0^\odot((V))^\times \rightarrow T_0^\otimes((V))^\times.$
Therefore, $T_0^\otimes((V))^\times$ is a central extension of $T_0^\odot((V))^\times.$ 

\subsection{Series indexed by $\mathbb{N}$-graded small categories}

We extend and extend the procedure used in \cite{Ma2013} for series indexed by $\N.$  
Let $(I,*)$ be a small category with neutral elements $e.$ By small category, we require that $*$ is associative, with neutral element(s) and that it is not necessarily total.
Let $A_i$ be a family of diffeological vector spaces indexed by $I.$ The family $\{\A_i; i \in I \}$ is equipped with a multiplication, associative and distributive with respect to addition in the vector spaces $\A_i,$ such that 
$$A_i . A_j \left\{ \begin{array}{ll} \subset A_{i*j} & \hbox{ if } i*j \hbox{ exists} \\
= 0 & \hbox{ otherwise. } \end{array} \right.$$and smooth.
Let $\mathcal{A}$ be the vector space of formal series of the type $ a = \sum_{i \in I} a_i $ with $ a_i \in {A}_i $
and such that, for each $k \in I $, there is a finite number of indexes $(i,j) \in I^2$ such that $i*j=k $ and $ai.a_j \neq 0_k.$ The $\N-$ grading on $I$ defines a degree and a valuation on $\A.$
Notice that, with such a definition, if ${e} \subset ord^{-1}(0)$ such that $e*e=e,$ $\A_e$ is an algebra.
From now, we assume $\A$ unital, and we note its unit element $1.$

\begin{Definition}
Let $I$ as above, such that, there is a $\N-$grading, that is,  a morphism of small categories $ord : I \rightarrow \N,$ such that the order of any unital element is $0.$  
Let $A_i$ be family of diffeological vector spaces indexed by $I.$
A diffeological vector subspace $\mathcal{A}_0 \subset \A = \left\{ \sum_{i \in I} a_i | a_i \in \mathcal{A}_i \right\}$ is called diffeological $I-$graded regular algebra if and only if it is equipped with a multiplication that extends $*$, associative and distributive with respect addition, and smooth.
\end{Definition}

\section{More examples} \label{pc}

\subsection{Examples of $h-$deformed pseudo-differential operators} \label{opd}
Let us describe a straightforward generalization of the groups described in \cite{Ma2013,MR2024,MR2025}. Let $E$ be a smooth vector bundle over a
compact manifold without boundary M. We denote by  $ Cl(M, E) $ (resp.  $ Cl^k (M, E)
$) the space of
 classical pseudo-differential operators (resp.
classical pseudo-differential operators of order k) acting on smooth
sections of $E$. We denote by $Cl^*(M,\mathbb{C}^n)$, resp.
 $Cl^{0,*}(M,\mathbb{C}^n),$ the groups of
the units of the algebras $Cl(M,\mathbb{C}^n),$ resp.
$Cl^{0}(M,\mathbb{C}^n)$.
Notice that $Cl^{0,*}(M,\mathbb{C}^n)$ is a CBH Lie group \cite{Glo}.
.
\begin{Definition}
Let $h$ be a formal parameter. 
We define the algebra of formal series 
$$Cl_h(M,E) = \left\{ \sum_{t \in \N^*} h^k a_k | \forall k \in \N^*, a_k \in Cl(M,E) \right\}.$$
\end{Definition}
This is obviously an algebra, graded by the order (the valuation) into the variable  $h.$ Thus, setting
$ \A_n = \left\{ h^n a_n | a_n \in Cl^n(M,E)\right\} ,$
we can set $\A = Cl_h(M,E).$ This is  a Fr\'echet vector space. Therefore, integrals of smooth paths exist, and the tangent space (which we carefully omitted to mention in this too short note, see \cite{GMW2024} for more details), is here the \emph{classical} tangent space of a vector space, and it can be identified with the vector space itself.  
Therefore, adapting the arguments of \cite{MR2024}, see e.g. \cite{GMW2024}, we get:
\begin{Cor}
The group $1 + Cl_h(M,E)$ is a Fr\'echet Lie group with Lie algebra $Cl_h(M,E).$ Moreover, the logarithmic equation 
\begin{equation} \label{eq:log}
g^{-1}dg = v, g(0)=1
\end{equation}
has an unique solution $g \in C^\infty(\R,1 + Cl_h(M,E))$ for a fixed path $v \in C^\infty(\R,Cl_h(M,E)).$ Finally, the map $v \mapsto g$ is smooth. 
\end{Cor}

\begin{rem}
Equation (\ref{eq:log}) defines the mapping $v \mapsto g$ which stands as a generalized exponential map for infinite dimensional groups. We have denoted by $\exp$ the exponential defined by formal series, but it can be either ill defined or unsufficient for technical features in differential geometry, which explains this second definition that is in general denoted by $\rm Exp$.
If $\rm Exp$ exists, the Lie group is called \emph{regular}.
\end{rem}
Let $Cl^{0,*}(M,E)$ be the Lie group of invertible pseudo-differential operators of order 0. This group is known to be a regular Lie group since Omori, but the most efficient proof is actually in \cite{Glo}, to our knowledge.
We remark a short exact sequence of Fr\'echet Lie groups:
$$ 0 \rightarrow 1 + Cl_h(M,E) \rightarrow Cl^{0,*}(M,E) + Cl_h(M,E) \rightarrow Cl^{0,*}(M,E) \rightarrow 0,$$  
which satisfies the necessary conditions \cite{KM} to state the following result:
\begin{Theorem}
$Cl^{0,*}(M,E) + Cl_h(M,E)$ is a regular Lie group with Lie algebra $Cl^{0}(M,E) + Cl_h(M,E).$
\end{Theorem}

\subsection{Series over cobordisms} \label{ss:cob} 
We describe here a setting where the indexes which live an a small category. 
This example recovers, passing to homotopy classes, the cobordism setting. 
We do not wish to consider homotopy invariant properties, and describe 
some kind of ``pseudo-cobordism''.
We now consider the set
$ Gr = \coprod_{m \in \N^*} Gr_m$
where $Gr_m$ is the set of m-dimensional compact connected oriented manifolds $M$, possibly with boundary, where the boundaries $\partial M$ are separated into two disconnected parts: the initial part $\alpha(M)$ and the final part $\beta(M).$
Then, we have a composition law $*$, called cobordism composition in the rest of the text, defined by the following relation:

\begin{Definition}
	Let $m \in \N^*.$ Let $M, M' \in Gr_m.$ Then $M'' = M*M' \in Gr_m$ exists if 
	\begin{enumerate}
		\item $\alpha(M) = \beta(M') \neq \emptyset,$ up to diffeomorphism
		\item $\alpha(M'')= \alpha(M')$
		and  $\beta(M'') = \beta(M)$
		\item $M''$ cuts into two pieces $M'' = M \cup M'$ with $M \cap M' = \alpha(M) = \beta(M').$
	\end{enumerate}
\end{Definition} 
\noindent
This  \textbf{cobordism composition} extends to embedded manifolds:
\begin{Definition}
	Let $N$ be a smooth (finite dimensinal) manifold.
	$$ Gr(N) = \coprod_{m \in \N^*} \coprod_{M \in Gr_m} Emb(M,N).$$
	where the notation $Emb(M,N)$ denotes the smooth manifold of smooth embeddings of $M$ into $N.$
\end{Definition}
\noindent
Notice that since $dim(N) < \infty,$ we have $m \leq dim(N).$
We recall that that $Gr(N)$ is a smooth manifold, since $Emb(M,N)$ is a smooth manifold \cite{KM}, and that $*$ is smooth because it is smooth for the underlying diffeologies. 

\begin{Definition}
	\begin{itemize}
		\item Let $ I =  \left( Gr \times \N^*\right) \coprod (\emptyset, 0),$ graded by the second component.  Assuming $\emptyset$ as a neutral element for $*$, we extend the cobordism composition into a composition, also noted $*$, defined as:
		$$ (M,p) * (M',p') = (M*M', p+ p')$$ when $M*M'$ is defined. We call \textbf{length} of $(M,p)$ the number $p.$
		\item
		Let $ I(N) =  \left( Gr(N) \times \N^*\right) \coprod (\emptyset, 0),$ graded by the second component.  Assuming $\emptyset$ as a neutral element for $*$, we extend the cobordism composition into a composition, also noted $*$, defined as
		$ (M,p) * (M',p') = (M*M', p+ p')$ when $M*M'$ is defined.
		\item Let $m \in \N^* .$ We note by $I_m$ and $I_m(N)$ the set of indexes based on $Gr_m$ and on $Gr_m(N)$ respectively
	\end{itemize}
\end{Definition}
\noindent
Let us now turn to $q-$deformed groups and algebras. For these definitions, the length number carries the natural grading for series.
Let $A$ be a Fr\'echet algebra. Let $m \in \N^*.$ Let 
$ \A_{I_m} = \left\{ \sum_{(M,n) \in I_m} q^na_{M,n} | a_{M;n} \in A \right\}$ 
{ and let } 
$\A_{I_m}(N) = \left\{ \sum_{(\phi,n) \in I_m(N)} q^na_{\phi,n} | a_{M;n} \in A \right\}. $
\begin{Theorem} \label{cob}
	Let $\Gamma \subset \coprod_{m \in \N^*} I_m$, resp. $\Gamma(N) \subset \coprod_{m \in \N^*} I_m(N)$, be a family of indexes, stable under $*,$ such that $\forall m \in \N^*,$
	\begin{enumerate}
		\item  $\forall m \in \N^*,$ $\Gamma \cap I_m$ is finite or, more generally;
		\item $\forall \gamma \in \Gamma,$ the set of pairs $(\gamma',\gamma'') \in \Gamma^2$ such that $\gamma = \gamma' *`\gamma''$ is finite.
	\end{enumerate}
	Then
	$ 1_A + \A_{\Gamma-\{(\emptyset,0)\}}$
	is a Fr\'echet Lie group with Lie algebra $A_\Gamma.$
	Moreover, for each regular diffeological Lie group $G $ with Lie algebra $\mathfrak{g}$ such that $G \subset A^*$ smoothly, $ G \oplus \A_{\Gamma - \{(\emptyset, 0)\}}$
	is a regular diffeological Lie group with Lie algebra $\mathfrak{g} \oplus \mathcal{A}_{\Gamma-\{(\emptyset,0)\}}.$
	Moreover, the results are the same replacing $\Gamma$ by $\Gamma(N).$
\end{Theorem} 	

The proof gathers all the techniques carefully described in the examples before this section. Therefore it can be left to the reader as an exercise.

\section{Perspectives}

The rather abstract constructions of section \ref{ss:cob} may serve as an indicative road sign to possible generalizations for applications, restricting the technical problems in order to serve as a pedagogical example. Indeed, more general refinements can be provided, by means of subtantial efforts in the control of the theoretical problems raised by the examples that one intends to consider. Such examples include the construction of formal series of measures in the framework of stochastic cosurfaces as developed in \cite{Ma2022}, by constructing explicitely such series. However, the necessarily diffeological framework for dealing safely with measures and probabilities actually has to be clearly precised, since many non equivalent options (diffeological or not) remain possible to deal with geometry on such objects. Moreover, the notion of regular Lie group in the diffeological category carries many particular problems that we did not have the place to develop here. An overview of selected technical particularities can be found in \cite{GMW2024,Les}.   
Within the framekork of probabilities, one can potentially find  such constructions in branching processes that have applications e.g. to morphogenesis problems, see e.g. \cite{BreI,BreII}, and more generaly any problem related to the asymptotic behavour of discrete stochastic processes. 
In a more deterministic perspective, we have to mention the ``old'' sewing and gluing formulas in conformal field theories refreshed in \cite{Mar2020}, as well as the Connes-Kreimer Hopf algebra built upon the binary rooted planar trees also present in Runge-Kutta methods \cite{BS2014,CK1998,HLW2002,murua2021}. One can also find an appeal to diffeological generalizations of (yet, already generalized) series in the outlook section of \cite{KT2018}, working along the lines of applications of (generalized) Taylor series in computer science, see e.g. \cite{ER2008,EW2025} for more classical approaches.


\begin{thebibliography}{99}







\bibitem{BH} Baez, J. C.; Hoffnung, A.E.; Convenient Categories of Smooth Spaces. {\em Transactions of 
	    the American Mathematical Society} 363, no. 11 (2011): 5789--5825. 

\bibitem{BKW2025} Batubenge, A.; Karshon, Y.; Watts, J.;
Diffeological, Fr\"olicher, and differential spaces. 
\emph{Rocky Mt. J. Math.} {\bf 55} no. 3, 637-670 (2025).


 \bibitem{BS12} Bogachev, V. I.; Smolyanov, O.G.; \emph{Topological Vector Spaces and Their Applications} Springer (2012).

\bibitem{BS2014} Bogfjellmo, G.; Schmeding, A.; The Lie group structure of the Butcher group;   \textit{Found. Comput. Math.}
 \textbf{17}, no 1, pp 127-159 (2017)


\bibitem{BreI} Bressloff, P. C.;
\emph{Stochastic processes in cell biology. Volume I.}  Springer  (2021). 

\bibitem{BreII} Bressloff, P. C.;
\emph{Stochastic processes in cell biology. Volume II.} Springer  (2022). 


\bibitem{CK1998} Connes, A.; Kreimer, D.; Hopf algebras, renormalization and noncommutative geometry. \emph{Comm. Math. Phys.} {\bf 199}  203-242 (1998).


\bibitem{ER2008} Ehrhard, T.; Regnier, L.; Uniformity and the taylor expansion of ordinary lambda-terms. \emph{ Theoret. Computer Sci.} {\bf 403} no 2, 347–372 (2008).

\bibitem{EW2025} Ehrhard, T.; Walch, A.;
Coherent Taylor expansion as a bimonad.
\emph{Math. Struct. Comput. Sci.} {\bf 35}, Paper no. e9, 122 p. (2025). 
\bibitem{ERMR} Eslami Rad, A.; Magnot, J.-P.; Reyes, E. G.; The Cauchy problem of the Kadomtsev-Petviashvili 
		hierarchy with arbitrary coefficient algebra. {\it J. Nonlinear Math. Phys.} 24:sup1 (2017), 103--120.



\bibitem{Glo} Gl\"ockner, H; Algebras whose groups of the units are
Lie groups \textit{Studia Math. } \textbf{153}, no2 (2002), 147-177

\bibitem{GMW2024} Goldammer, N.; Magnot, J-P.; Welker, K.; On diffeologies from infinite dimensional geometry to PDE constrained optimization. \emph{Contemp. Math.} {\bf 794}, 1-48 (2024).

\bibitem{HLW2002} Hairer, E.; Lubich, C.; Wanner, G.;
\emph{Geometric numerical integration. Structure-preserving algorithms for ordinary differential equations.} 
Springer Series in Computational Mathematics. 31.  (2002).

\bibitem{Igdiff} Iglesias-Zemmour, P.
\textit{Diffeology} 
Mathematical Surveys and Monographs \textbf{185} AMS  (2013).

\bibitem{KT2018} Kerjean, M.; Tasson, C.;
Mackey-complete spaces and power series – a topological model of differential linear logic. 
\emph{Math. Struct. Comput. Sci.} {\bf 28} no. 4, 472-507 (2018).  

\bibitem{KM} Kriegl, A.; Michor, P.W.; \textit{The convenient setting
for global analysis} Math. surveys and monographs \textbf{53}, American
Mathematical society, Providence, USA. (2000)


\bibitem{Les} Leslie, J.; On a Diffeological Group Realization of
certain Generalized symmetrizable Kac-Moody Lie Algebras \textit{J.
Lie Theory} \textbf{13} (2003), 427-442








\bibitem{Ma2006-3} Magnot, J-P.; Diff\'eologie sur le fibr\'e d'holonomie d'une connexion en dimension infinie
\textit{C. R. Math. Acad. Sci. Soc. R. Can.} \textbf{28} no4 121--128 (2006)


\bibitem{Ma2013} Magnot, J-P.; Ambrose-Singer theorem on diffeological bundles and complete integrability of the KP equation; \textit{Int. J. Geom. Meth. Mod. Phys.} \textbf{ 10}, No. 9, Article ID 1350043, 31 p. (2013). (2013)


\bibitem{MR2016} Magnot, J-P.; Reyes, E. G.; 
Well-posedness of the Kadomtsev-Petviashvili hierarchy, Mulase factorization, and Fr\"olicher Lie groups. 
\emph{Ann. Henri Poincar\'e} {\bf 21} no. 6, 1893-1945 (2020). 
\bibitem{Ma2022} Magnot, J-P.;On stochastic cosurfaces and topological quantum field theories. 
\emph{Methods Funct. Anal. Topol.} {\bf 28} no. 3, 242-258 (2022). 
\bibitem{MR2024} Magnot, J-P.; Reyes, E.G.;
On the Cauchy problem for a Kadomtsev-Petviashvili hierarchy on non-formal operators and its relation with a group of diffeomorphisms. 
\emph{Dyn. Partial Differ. Equ.} {\bf 21} no. 3, 235-260 (2024). 
\bibitem{MR2025} Magnot, J-P.; Reyes, E.G.; Kadomtsev-Petviashvili hierarchies with non-formal pseudo-differential operators, non-formal solutions, and a Yang-Mills–like formulation. \emph{Phys. Lett. B} {\bf 867}  139589 (2025)
\bibitem{Mar2020} Marolf, D.;
CFT sewing as the dual of AdS cut-and-paste. 
\emph{J. High Energy Phys.} {\bf 2020} no. 2, Paper No. 152 (2020). 

\bibitem{murua2021} Murua, A.;
From Runge-Kutta methods to Hopf algebras of rooted trees. 
Makhlouf, Abdenacer (ed.), Algebra and applications 2. Combinatorial algebra and Hopf algebras.  179-219 (2021). 
\bibitem{Sou} Souriau, J.M.; Un algorithme g\'en\'erateur de structures quantiques; 
\textit{Ast\'erisque}, Hors S\'erie, (1985) 341-399 


\bibitem{Wa} Watts, J.; \textit{Diffeologies, differentiable spaces
and symplectic geometry} PhD thesis arXiv:1208.3634v1

\end{thebibliography}
\end{document}